\documentclass[a4paper,oneside,11pt]{article}%
\usepackage{makeidx}
\usepackage[english]{babel}
\usepackage{amsmath}
\usepackage{amsfonts}
\usepackage{amssymb}
\usepackage{stmaryrd}
\usepackage{graphicx}
\usepackage{mathrsfs}
\usepackage[colorlinks,linkcolor=red,anchorcolor=blue,citecolor=blue,urlcolor=blue]{hyperref}
\usepackage[symbol*,ragged]{footmisc}
\providecommand{\U}[1]{\protect\rule{.1in}{.1in}}
\providecommand{\U}[1]{\protect \rule{.1in}{.1in}}

\newtheorem{theorem}{Theorem}[section]

\newtheorem{lemma}[theorem]{Lemma}

\newtheorem{proposition}[theorem]{Proposition}

\newenvironment{proof}[1][Proof]{\noindent \textbf{#1.} }{\  \rule{0.5em}{0.5em}}
\numberwithin{equation}{section}

\begin{document}

\title{On Two Diophantine Inequalities Over Primes}

\author{Min Zhang\footnotemark[1]\,\,\,\, \, \& \,\, Jinjiang Li\footnotemark[2]    \vspace*{-4mm} \\
$\textrm{\small Department of Mathematics, China University of Mining and Technology}^{*\,\dag}$
                    \vspace*{-4mm} \\
     \small  Beijing 100083, P. R. China  }

\footnotetext[2]{Corresponding author. \\
    \quad\,\, \textit{ E-mail addresses}: \href{mailto:min.zhang.math@gmail.com}{min.zhang.math@gmail.com} (M. Zhang),
     \href{mailto:jinjiang.li.math@gmail.com}{jinjiang.li.math@gmail.com} (J. Li).   }

\date{}
\maketitle


{\textbf{Abstract}}: Let~$1<c<37/18,\,c\neq2$~and~$N$~be a sufficiently large real number. In this paper, we prove
that, for almost all~$R\in(N,2N],$~
the Diophantine inequality~$|p_1^c+p_2^c+p_3^c-R|<\log^{-1}N$~is solvable in primes~$p_1,\,p_2,\,p_3.$~
Moreover, we also investigate the problem of six primes and prove that the
Diophantine inequality~$|p_1^c+p_2^c+p_3^c+p_4^c+p_5^c+p_6^c-N|<\log^{-1}N$~is solvable
in primes~$p_1,\,p_2,\,p_3,\,p_4,\,p_5,\,p_6$~for sufficiently large real number~$N$.~


{\textbf{MR(2010) Subject Classification}}: 11L20, 11P05, 11P55

\section{Introduction and main result}

    In 1952,~Piatetski-Shapiro \cite{Piatetski-Shapiro} considered the following analogue of the
Waring-Goldbach problem.
   Assume that~$c>1$~is not an integer and let~$\varepsilon$~be a positive number. If~$r$~is a sufficiently large
integer~(depending only on~$c$),~then the inequality
\begin{equation}\label{Introduction-1}
   |p_1^c+p_2^c+\cdots+p_r^c-N|<\varepsilon
\end{equation}
has a solution in prime numbers $p_1,p_2,\cdots,p_r$ for sufficiently large $N$. More precisely, if the least $r$ such that (\ref{Introduction-1}) has a solution in prime numbers for every $\varepsilon>0$ and $N>N_0(c,\varepsilon)$ is denoted by $H(c),$~then it is proved in \cite{Piatetski-Shapiro} that
\begin{equation*}
  \limsup_{c\to\infty}\frac{H(c)}{c\log c}\leqslant 4.
\end{equation*}

  In \cite{Piatetski-Shapiro}, Piatetski-Shapiro also proved that if $1<c<3/2,$ then $H(c)\leqslant5.$ The upper
bound $3/2$ for $c$ was improved successively to
\begin{equation*}
  \frac{14142}{8923}=1.5848\cdots, \frac{1+\sqrt{5}}{2}=1.6180\cdots, \frac{81}{40}=2.025,\frac{108}{53}=2.0377\cdots, 2.041
\end{equation*}
by Zhai and Cao \cite{Zhai-Cao-1}, Garaev \cite{Garaev}, Zhai and Cao \cite{Zhai-Cao-2}, Shi and Liu \cite{Shi-Liu},
Baker and Weingartner \cite{Baker-Weingartner-1}, respectively.

  On the other hand, the Vinogradov-Goldbach theorem \cite{I. M. Vinogradov} suggests that at least for~$c$~close to~$1,$~one should expect~$H(c)\leqslant3.$~The first result in this direction was obtained by D. I. Tolev~\cite{D. I. Tolev}, who showed that the inequality
\begin{equation}\label{Introduction-2}
   |p_1^c+p_2^c+p_3^c-N|<\varepsilon
\end{equation}
with~$\varepsilon=N^{-(1/c)(15/14-c)}\log^9N$~is solvable in primes~$p_1,p_2,p_3,$~provided
that~$1<c<15/14$~and~$N$~is sufficiently large. Later, Tolev's range was enlarged to $1<c<13/12$ in Cai~\cite{Cai-1}, $1<c<11/10 $ in Cai~\cite{Cai-2} and Kumchev-Nedeva~\cite{Kumchev-Nedeva}~independently,~$1<c<237/214$~in~Cao~and~Zhai~\cite{Cao},~$1<c<61/55$~in Kumchev~\cite{Kumchev},~$1<c<10/9$ in Baker and~Weingartner~\cite{Baker-Weingartner-2}.

  Laporta \cite{Laporta} studied the corresponding binary problem, which can be viewed as an inequality analogue of the Goldbach's conjecture for even numbers. Suppose~$1<c<15/14$~fixed,~$N$~a large real number
  and~$\varepsilon=N^{1-15/(14c)}\log^8N$. Then Laporta proved that the inequality
\begin{equation}\label{Introduction-3}
  |p_1^c+p_2^c-R|<\varepsilon
\end{equation}
is solvable for all $R\in(N,2N]\setminus \mathfrak{A}$ with
$|\mathfrak{A}|\ll N\exp\left(-\frac{1}{3}\left(\frac{\log N}{c}\right)^{1/5}\right).$
Zhai and Cao \cite{Zhai-Cao-3} improved Laporta's \cite{Laporta} result and proved for $1<c<43/36$ fixed and for all
$R\in(N,2N]\setminus\mathfrak{A}$ with $|\mathfrak{A}|\ll N\exp\left(-\frac{1}{3}\left(\frac{\log N}{c}\right)^{1/5}\right)$, the inequality (\ref{Introduction-3}) is solvable with primes $p_1,p_2\leqslant N^{1/c}$
and $\varepsilon=N^{1-43/(36c)}.$

   In this paper we shall prove the following two Theorems.

\begin{theorem}\label{Theorem-exceptional}
   Let $1<c<37/18,\,c\neq2\,\textrm{and~$N$~be a sufficiently large real number}.$ Then for all $R\in(N,2N]\setminus\mathfrak{P}$ with
   $$|\mathfrak{P}|\ll N\exp\left(-\frac{2}{15}\left(\frac{1}{c}\log\frac{2N}{3}\right)^{1/5}\right),$$
    the inequality
  \begin{equation}\label{Introduction-5}
     |p_1^c+p_2^c+p_3^c-R|<\log^{-1}N
  \end{equation}
  is solvable in three prime variables~$p_1,p_2,p_3,$~where~$\eta$~is sufficiently small positive number.
\end{theorem}

\noindent
\textbf{Remark.}~The best result up to date for~$H(c)\leqslant3$~was obtained by Baker and~Weingartner~\cite{Baker-Weingartner-2}, who prove that~$1<c<10/9.$~From Theorem~\ref{Theorem-exceptional}, one can expect that the range of~$c$~for~$H(c)\leqslant3$~should be improved to~$1<c<37/18,\,c\neq2.$~Moreover, it is conjectured  that the range of~$c,$~which holds for~$H(c)\leqslant3,$~is~$1<c<3,\,c\neq2.$~Therefore, the range of~$c$~for~$H(c)\leqslant3$~has huge space to improve, though such a strong conjecture is out of reach at present.

\begin{theorem}\label{Theorem-1}
   Suppose that $1<c<37/18,\,c\neq2,$ then there exists a number $N_0(c)$ such that for each real number $N>N_0(c)$
   the inequality
   \begin{equation}\label{Introduction-4}
      |p_1^c+p_2^c+p_3^c+p_4^c+p_5^c+p_6^c-N|<\log^{-1}N
   \end{equation}
   is solvable in six prime variables~$p_1,p_2,p_3,p_4,p_5,p_6.$~
\end{theorem}

\textbf{Notation.} Throughout this paper,~$N$~always denotes a sufficiently large real number;~$\eta$~always denotes an arbitrary small positive constant, which may not be the same at different occurances;~$p$~always denotes a prime number;~
$n\sim N$~means~$N<n\leqslant 2N$;~$X\asymp N^{1/c},$~which is determined during each proof of the Theorems;~$\tau=X^{1-c-\eta},\,\varepsilon=\log^{-2}X,\,K=\log^5X;$~$\Lambda(n)$~denotes von Mangold's function;
~$\mu(n)$~denotes M\"{o}bius function;~$e(x)=e^{2\pi ix};$~$\mathscr{L}=\log X;\,E=\exp(-\mathscr{L}^{1/5}),$~
\begin{equation*}
P=\bigg(\frac{2}{E^2}\bigg)^{1/3}\mathscr{L},\,\,\,\,S(x)=\sum_{X/2<p\leqslant X}\log p\cdot e(p^cx),\,\,\,\,\,I(x)=\int_{\frac{X}{2}}^X e(t^cx)\mathrm{d}t.
\end{equation*}

\section{Preliminary Lemmas }

\begin{lemma}\label{xiaobei}
   Let $a,b$ be real numbers, $0<b<a/4,$ and let $k$ be a positive integer. There exists a function $\varphi(y)$
   which is $k$ times continuously differentiable and such that

  \begin{equation*}
    \left\{
      \begin{array}{cll}
          \varphi(y)=1,    & &  \textrm{for \quad} |y|\leqslant a-b, \\
          0<\varphi(y)<1,  & &  \textrm{for \quad} a-b<|y|< a+b, \\
          \varphi(y)=0,    & &  \textrm{for \quad} |y|\geqslant a+b,
      \end{array}
    \right.
  \end{equation*}
  and its Fourier transform
   \begin{equation*}
      \Phi(x)=\int_{-\infty}^{+\infty} e(-xy)\varphi(y)\mathrm{d}y
   \end{equation*}
   satisfies the inequality
   \begin{equation}
      \left|\Phi(x)\right|\leqslant\min\left(2a,\frac{1}{\pi|x|},\frac{1}{\pi|x|}\left(\frac{k}{2\pi|x|b}\right)^k\right).
   \end{equation}
\end{lemma}
\begin{proof}
 See Piatetski-Shapiro~\cite{Piatetski-Shapiro} or Segal~\cite{Segal}.
\end{proof}

\begin{lemma}\label{Van der Corput}
    Let $G,\,F$ be twice differentiable on $[A,B],\, |G(x)|\leqslant H,\, G/F'$ monotonic. If $F'\geqslant K>0$ on $[A,B],$ then
      \begin{equation*}
         \int_{A}^BG(x)e(F(x))\mathrm{d}x\ll HK^{-1}.
      \end{equation*}
\end{lemma}
\begin{proof}
  See Titchmarsh \cite{Titchmarsh}, Lemma 4.3.
\end{proof}

\begin{lemma}\label{Robert-lemma}
  Suppose $M>1,c>1,\gamma>0.$ Let $\mathscr{A}(M;c,\gamma)$ denote the number of solutions of the inequality
 \begin{eqnarray*}
   |n_1^c+n_2^c-n_3^c-n_4^c|<\gamma,\qquad M<n_1,n_2,n_3,n_4\leqslant2M,
 \end{eqnarray*}
 then
 \begin{equation*}
    \mathscr{A}(M;c,\gamma)\ll(\gamma M^{4-c}+M^2)M^\eta.
 \end{equation*}
\end{lemma}
\begin{proof}
 See Robert and Sargos \cite{Robert-Sargos}, Theorem 2.
\end{proof}

\begin{lemma}\label{6-power-zhuxiang-low-bound}
   For $1<c<3,c\neq2,$
   we have
     \begin{equation*}
        \int_{-\infty}^{+\infty} I^6(x)e(-xN)\Phi(x)\mathrm{d}x\gg\varepsilon X^{6-c}.
     \end{equation*}
\end{lemma}

\begin{proof}
 Denote the above integral by~$\mathscr{H}$.~We have
\begin{equation*}
   \mathscr{H}:=\int_{\frac{X}{2}}^X\cdots\int_{\frac{X}{2}}^X\int_{-\infty}^{+\infty}
   e\big((t_1^c+t_2^c+\cdots+t_6^c-N)x\big)\Phi(x)\mathrm{d}x\mathrm{d}t_1\cdots\mathrm{d}t_6.
\end{equation*}
The change of the order of integration is legitimate because of the absolute convergence of the integral. From Lemma \ref{xiaobei} 
with $a=9\varepsilon/10,\,b=\varepsilon/10$, by using the Fourier inversion formula we get
\begin{equation*}
 \mathscr{H}=\int_{\frac{X}{2}}^X\cdots\int_{\frac{X}{2}}^X\varphi(t_1^c+t_2^c+\cdots+t_6^c-N)\mathrm{d}t_1\cdots\mathrm{d}t_6.
\end{equation*}
By the definition of $\varphi(y)$ we get
\begin{equation*}
   \mathscr{H}\geqslant \mathop{\int_{\frac{X}{2}}^X\cdots\int_{\frac{X}{2}}^X}_{|t_1^c+\cdots+t_6^c-N|<\frac{4}{5}\varepsilon}
   \mathrm{d}t_1\cdots\mathrm{d}t_6 \geqslant\int_{\lambda X}^{\mu X}\cdots\int_{\lambda X}^{\mu X}
   \left(\int_{\mathfrak{M}}\mathrm{d}t_6\right)\mathrm{d}t_1\cdots\mathrm{d}t_5,
\end{equation*}
where $\lambda$ and $\mu$ are real numbers such that
\begin{equation*}
   \frac{1}{2}<\left(\frac{4}{5}\right)^{1/c}<\lambda<\mu<\left(1-\frac{1}{5}\cdot\frac{1}{2^c}\right)^{1/c}<1
\end{equation*}
and
\begin{eqnarray*}
  \mathfrak{M}& = & \left[\frac{X}{2},X\right]\bigcap
               \left[\left(N-\frac{4\varepsilon}{5}-t_1^c-\cdots-t_5^c\right)^{1/c},
               \left(N+\frac{4\varepsilon}{5}-t_1^c-\cdots-t_5^c\right)^{1/c}\right] \\
    & = &      \left[\left(N-\frac{4\varepsilon}{5}-t_1^c-\cdots-t_5^c\right)^{1/c},
               \left(N+\frac{4\varepsilon}{5}-t_1^c-\cdots-t_5^c\right)^{1/c}\right].
\end{eqnarray*}
Thus by the mean-value theorem we have
\begin{eqnarray*}
  \mathscr{H}\gg \varepsilon\int_{\lambda X}^{\mu X}\cdots\int_{\lambda X}^{\mu X}(\xi_{t_1,t_2,t_3,t_4,t_5})^{1/c-1}
  \mathrm{d}t_1\cdots\mathrm{d}t_5,
\end{eqnarray*}
where $\xi_{t_1,t_2,t_3,t_4,t_5}\asymp X^c.$ Therefore, $\mathscr{H}\gg\varepsilon X^{6-c},$ which proves the lemma.
\end{proof}

\begin{lemma}\label{Max-value-lemma}
   We have
   \begin{equation*}
       A=\max_{R'\in(N,2N]}
       \int_{N}^{2N}\bigg|\int_{\tau<|x|< K}e\left((R-R')x\right)\mathrm{d}x\bigg|\mathrm{d}R
           \ll \mathscr{L}.
   \end{equation*}
\end{lemma}
\begin{proof}
  See Laporta~\cite{Laporta}, Lemma 1.
\end{proof}

 Let $\Omega_1$ and $\Omega_2$ be measurable subsets of $\mathbb{R}^n.$ Let
 \begin{equation*}
    \|f\|_j=\bigg(\int_{\Omega_j}|f(y)|^2\mathrm{d}y\bigg)^{\frac{1}{2}},\quad
    \langle f,g\rangle_j=\int_{\Omega_j}f(y)\overline{g(y)}\mathrm{d}y\,\,\,(j=1,2),
 \end{equation*}
be the usual norm and inner product in $L^2(\Omega_j,\mathbb{C})$, respectively.

\begin{lemma}\label{Functional-lemma}
   Let $c\in L^2(\Omega_1,\mathbb{C}),\,\xi\in L^2(\Omega_2,\mathbb{C}),$ and let $\omega$ be a measurable
   complex valued function on $\Omega_1\times\Omega_2$ such that
   \begin{equation*}
     \sup_{x\in \Omega_1}\int_{\Omega_2}|\omega(x,y)|\mathrm{d}y<+\infty,\quad
     \sup_{y\in \Omega_2}\int_{\Omega_1}|\omega(x,y)|\mathrm{d}x<+\infty.
   \end{equation*}
   Then we have
   \begin{equation*}
    \bigg|\int_{\Omega_1}c(x)\langle\xi,\omega(x,\cdot)\rangle_2\mathrm{d}x\bigg|
    \leqslant \|\xi\|_2\|c\|_1\bigg(\sup_{x'\in\Omega_1}\int_{\Omega_1}\left|\langle
             \omega(x,\cdot),\omega(x',\cdot)\rangle_2\right|\mathrm{d}x\bigg)^{\frac{1}{2}}.
   \end{equation*}
\end{lemma}
\begin{proof}
 See Laporta~\cite{Laporta}, Lemma 2.
\end{proof}

\begin{lemma}\label{square-mean-value}
   For~$1<c<37/18,\,c\neq2,$~we have
   \begin{eqnarray}
       & &  \int_{-\tau}^\tau |S(x)|^2\mathrm{d}x\ll X^{2-c}\log^3X, \\
       & &  \int_{-\tau}^\tau |I(x)|^2\mathrm{d}x\ll X^{2-c}\log X.
   \end{eqnarray}
\end{lemma}
\begin{proof}
    See Tolev~\cite{D. I. Tolev}, Lemma 7. Although in Tolev's paper,~$c$~is in the range~$(1,15/14),$~it can be easily seen
    that his lemma is true for~$c\in(1,2)\cup(2,3)$, and so do his Lemmas~$11-14.$~In fact, the proofs of
    Lemma~$7\,\textrm{and Lemmas}\,11-14$~in~\cite{D. I. Tolev}~have nothing to do with the range of~$c.$~
\end{proof}

\begin{lemma}\label{S(x)=I(x)+error}
   For~$1<c<37/18,\,c\neq2,\,|x|\leqslant\tau,$~then
  \begin{equation*}\label{S(x)ti-huan-I(x)}
     S(x)=I(x)+O\left(Xe^{-(\log X)^{1/5}}\right).
  \end{equation*}
\end{lemma}
\begin{proof}
    See Tolev~\cite{D. I. Tolev}, Lemma 14.
\end{proof}

\begin{lemma}\label{S(x),I(x)-4-mean-square}
  For~$1<c<37/18,\,c\neq2,$~we have
   \begin{eqnarray}
       &  &   \int_{-\tau}^{\tau}|S(x)|^4\mathrm{d}x\ll X^{4-c}\log ^5X,  \label{(i)}   \\
       &  &   \int_{-\tau}^{\tau}|I(x)|^4\mathrm{d}x\ll X^{4-c}\log ^5X.   \label{(ii)}
   \end{eqnarray}
\end{lemma}
\begin{proof}
 We only prove (\ref{(i)}). Inequality (\ref{(ii)}) can be proved likewise.

 We have
\begin{eqnarray}\label{S(x)-I(x)-4-mean-1}
   \int_{\tau}^{\tau}|S(x)|^4\mathrm{d}x
      & = & \sum_{\frac{X}{2}<p_1,\,p_2,\,p_3,\,p_4\leqslant X}(\log p_1)\cdots(\log p_4)
            \int_{\tau}^{\tau} e\left((p_1^c+p_2^c-p_3^c-p_4^c)x\right) \mathrm{d}x
                    \nonumber   \\
      & \ll & \sum_{\frac{X}{2}<p_1,\,p_2,\,p_3,\,p_4\leqslant X}(\log p_1)\cdots(\log p_4)
              \cdot\min\bigg(\tau,\frac{1}{|p_1^c+p_2^c-p_3^c-p_4^c|}\bigg)
                    \nonumber   \\
      & \ll & U\tau\log^4X+V\log^4X,
\end{eqnarray}
where
\begin{equation*}
   U=\sum_{\substack{\frac{X}{2}<n_1,\,n_2,\,n_3,\,n_4\leqslant X\\
            |n_1^c+n_2^c-n_3^c-n_4^c|\leqslant 1/\tau}}1\,\,,
      \qquad V=\sum_{\substack{\frac{X}{2}<n_1,\,n_2,\,n_3,\,n_4\leqslant X\\
            |n_1^c+n_2^c-n_3^c-n_4^c|> 1/\tau}}
            \frac{1}{|n_1^c+n_2^c-n_3^c-n_4^c|}.
\end{equation*}
We have
\begin{eqnarray*}
    U & \ll & \sum_{\frac{X}{2}<n_1\leqslant X} \sum_{\frac{X}{2}<n_2\leqslant X} \sum_{\frac{X}{2}<n_3\leqslant X}
          \sum_{\substack{\frac{X}{2}<n_4\leqslant X \\
          (n_1^c+n_2^c-n_3^c-1/\tau)^{1/c}\leqslant n_4\leqslant (n_1^c+n_2^c-n_3^c+1/\tau)^{1/c} \\
                           n_1^c+n_2^c-n_3^c\asymp X^c }} 1   \\
    & \ll &  \sum_{\substack{\frac{X}{2}<n_1,\,n_2,\,n_3\leqslant X\\n_1^c+n_2^c-n_3^c\asymp X^c}}
    \left(1+(n_1^c+n_2^c-n_3^c+1/\tau)^{1/c}-(n_1^c+n_2^c-n_3^c-1/\tau)^{1/c}\right)
\end{eqnarray*}
and by the mean-value theorem
\begin{equation}\label{S(x)-I(x)-4-mean-2}
   U\ll X^3+\frac{1}{\tau}X^{4-c}.
\end{equation}
Obviously,~$V\leqslant\displaystyle\sum_{\ell}V_\ell,$~where
\begin{equation}\label{S(x)-I(x)-4-mean-3}
   V_\ell=\sum_{\substack{\frac{X}{2}<n_1,\,n_2,\,n_3,\,n_4\leqslant X\\
                    \ell<|n_1^c+n_2^c-n_3^c-n_4^c|\leqslant2\ell}}
                     \frac{1}{|n_1^c+n_2^c-n_3^c-n_4^c|}
\end{equation}
and $\ell$ takes the values $\frac{2^k}{\tau},\,k=0,1,2,\cdots,$ with $\ell\ll X^c.$
Then, we have
\begin{eqnarray*}
  V_\ell & \ll  & \frac{1}{\ell} \sum_{\substack{\frac{X}{2}<n_1,\,n_2,\,n_3,\,n_4\leqslant X \\
                  (n_1^c+n_2^c-n_3^c+\ell)^{1/c}\leqslant n_4\leqslant (n_1^c+n_2^c-n_3^c+2\ell)^{1/c}\\
                  n_1^c+n_2^c-n_3^c\asymp X^c }}1.   \\
\end{eqnarray*}
For $\ell\geqslant1/\tau$ and $X/2<n_1,\,n_2,\,n_3\,\leqslant X$~with~$n_1^c+n_2^c-n_3^c\asymp X^c$,
it is to see that
\begin{equation*}
  (n_1^c+n_2^c-n_3^c+2\ell)^{1/c}-(n_1^c+n_2^c-n_3^c+\ell)^{1/c}>1.
\end{equation*}
Hence,
\begin{equation}\label{S(x)-I(x)-4-mean-4}
  V_\ell \ll\frac{1}{\ell}\sum_{\substack{\frac{X}{2}<n_1,\,n_2,\,n_3 \leqslant X\\n_1^c+n_2^c-n_3^c\asymp X^c }}
           \left((n_1^c+n_2^c-n_3^c+2\ell)^{1/c}-(n_1^c+n_2^c-n_3^c+\ell)^{1/c}\right)\ll X^{4-c}
\end{equation}
 by the mean-value theorem.

 The conclusion follows from formulas (\ref{S(x)-I(x)-4-mean-1})-(\ref{S(x)-I(x)-4-mean-4}).
\end{proof}

\begin{lemma}\label{S(x)estimate-1<c<2}
   If $1<c<2,\tau\leqslant|x|\leqslant K,$ then we have
   \begin{equation*}
      S(x)\ll X^{\frac{6+c}{8}+\eta}+X^{\frac{14}{15}+\eta}.
   \end{equation*}
\end{lemma}
\begin{proof}
  See Zhai and Cao~\cite{Zhai-Cao-1}, Lemma 7.
\end{proof}

\begin{lemma}\label{Fouvry-Iwaniec-lemma}
  Let~$N,Q\geqslant1$~and~$z_n\in\mathbb{C}.$~Then
   \begin{equation*}
       \bigg|\sum_{n\sim N}z_n\bigg|^2  \leqslant\left(2+\frac{N}{Q}\right)
       \sum_{|q|<Q}\left(1-\frac{|q|}{Q}\right)\sum_{N<n+q,n-q\leqslant2N}z_{n+q}\overline{z_{n-q}}.
    \end{equation*}
\end{lemma}
\begin{proof}
  See Fouvry and Iwaniec~\cite{Fouvry-Iwaniec}, Lemma 2.
\end{proof}

\begin{lemma}\label{Graham-Kolesnik}
  Suppose that
  \begin{equation*}
       L(H)=\sum_{i=1}^mA_iH^{a_i}+\sum_{j=1}^nB_jH^{-b_j},
  \end{equation*}
  where~$A_i,\,B_j,\,a_i\,\textrm{and}\,\,b_j$~are positive. Assume that~$H_1\leqslant H_2.$~Then there is some~$\mathscr{H}$~with~$H_1\leqslant\mathscr{H}\leqslant H_2$~and
   \begin{equation*}
      L(\mathscr{H}) \ll \sum_{i=1}^{m}A_iH_1^{a_i}+\sum_{j=1}^{n}B_jH_2^{-b_j}
                         +\sum_{i=1}^m\sum_{j=1}^{n}\big(A_i^{b_j}B_j^{a_i}\big)^{1/(a_i+b_j)}.
   \end{equation*}
   The implied constant depends only on~$m$~and~$n.$~
\end{lemma}
\begin{proof}
 See Graham and Kolesnik~\cite{Graham-Kolesnik},~Lemma~2.4.~
\end{proof}

 For the sum of the form
\begin{equation*}
  \sum_{M<m\leqslant M_1}\sum_{N<n\leqslant N_1}a_mb_ne(xm^cn^c)
\end{equation*}
with
\begin{equation*}
  MN\sim X,\,\,M<M_1\leqslant 2M,\,\,N<N_1\leqslant2N,\,\,a_m\ll X^{\eta},\,\,b_n\ll X^{\eta}
\end{equation*}
for every fixed~$\eta,$~it is usually called a ``~Type I~" sum, denoted by~$S_I(M,N),$~if~$b_n=1$~or~$b_n=\log n;$~otherwise it is called a ``~Type II~" sum, denoted by~$S_{II}(M,N).$~

\begin{lemma}\label{Baker-lemma}
  Let~$\alpha,\beta\in\mathbb{R},\,\alpha\neq0,1,2,\,\beta\neq0,1,2,3.$~For~$F\gg MN^2$~and~$N\geqslant M\geqslant1,$~we have
  \begin{eqnarray*}
   S_{II}(M,N) & = &  \sum_{m\sim M}\sum_{n\sim N}a_mb_ne\Big(F\frac{m^\alpha n^\beta}{M^\alpha N^\beta}\Big) \\
         & \ll_{\alpha,\beta,\eta} & (MN)^\eta  \big(M^{7/8}N^{13/16}F^{1/16}+M^{93/104}N^{23/26}F^{1/26}  \\
                &  &   +M^{467/512}N^{65/64}F^{-1/128}+M^{65/72}N\big).
   \end{eqnarray*}
\end{lemma}
\begin{proof}
 See Baker and Weingartner~\cite{Baker-Weingartner-1},~Theorem~1.~
\end{proof}

In the rest of this section, we always suppose~$2<c<33/16,\,\delta=c/2-1+\eta,\,F=|x|X^c,\,\tau\leqslant|x|\leqslant K.$~Obviously, we have~$X^{1-\eta}\ll F\ll KX^c.$~

\begin{lemma}\label{Type-I}
   Suppose~$2<c<37/18,\,b_n\ll 1.$~If there holds~$M\gg X^{1-72\delta/7}$, then we have
   \begin{equation*}
      S_{I}(M,N)=\sum_{m\sim M}\sum_{n\sim N}b_ne(xm^cn^c)\ll X^{1-\delta}.
   \end{equation*}
\end{lemma}
\begin{proof}
  Let~$f(m)=xm^cn^c.$~Then we have~$|f^{(j)}(m)|\asymp (FM^{-1})M^{1-j}$~for~$j=1,\cdots,6.$~By the method of exponent pairs, we get
  \begin{eqnarray*}
   S_{I} 
       & \ll & \sum_{n\sim N} \bigg|\sum_{m\sim M}e(xm^cn^c) \bigg| \\
       & \ll & N\big(MF^{-1}+(FM^{-1})^\kappa M^\lambda\big)  \\
       & \ll & XF^{-1}+F^{\kappa}M^{\lambda-\kappa}N  \\
       & \ll & X^\eta+K^\kappa X^{\kappa c}X^{\lambda-\kappa}N^{1+\kappa-\lambda} \\
       & \ll & (\log X)^{5\kappa}X^{\kappa c+\lambda-\kappa+(1+\kappa-\lambda)/3}.
\end{eqnarray*}
The last step is due to the fact that~$N\asymp XM^{-1}\ll X^{72\delta/7}\ll X^{1/3}.$~
Taking the exponent pair~$(\kappa,\lambda)=A^3(1/2,1/2)=(1/30,26/30),$~then we obtain
\begin{equation*}
   S_{I}(M,N)\ll X^{1-\delta}
\end{equation*}
by noting that~$2<c<37/18.$~
\end{proof}

\begin{lemma}\label{Type-II}
   Suppose~$2<c<37/18,\,a_m\ll 1,\,b_n\ll 1.$~If there holds~$X^{72\delta/7}\ll M\ll X^{1/2}$, then we have
    \begin{equation*}
        S_{II}(M,N)=\sum_{m\sim M}\sum_{n\sim N}a_mb_ne(xm^cn^c)\ll X^{1-\delta}.
    \end{equation*}
\end{lemma}
\begin{proof}
    Take a suitable~$F_0\geqslant MN^2,$~whose value will be determined later during the following discussion.
    If~$F\geqslant F_0,$~according to Theorem~$1$~of~Baker~and~Weingartner~\cite{Baker-Weingartner-1},~we obtain
   \begin{eqnarray*}
     X^{-\eta}\cdot S_{II}(M,N)
             & \ll & M^{7/8}N^{13/16}F^{1/16}+M^{93/104}N^{23/26}F^{1/26}      \\
             &     &     +M^{467/512}N^{65/64}F^{-1/128} +M^{65/72}N           \\
             &  =: & \mathscr{I}_1+\mathscr{I}_2+\mathscr{I}_3+\mathscr{I}_4.
   \end{eqnarray*}
    Noting that if there holds~$X^{72\delta/7}\ll M \ll X^{1/2}$, we obtain
    \begin{equation*}
        \mathscr{I}_1\ll X^{1-\delta},\,\,\mathscr{I}_2\ll X^{1-\delta},\,\,\mathscr{I}_4\ll X^{1-\delta}.
    \end{equation*}
   Therefore, for the case~$F\gg F_0\gg MN^2,$~we get
   \begin{equation}\label{Type-bound-1}
       S_{II}(M,N)\ll X^{1-\delta}+M^{467/512}N^{65/64}F_0^{-1/128}.
   \end{equation}
   Next, we consider the case~$X^{1-\eta}\ll F\ll F_0.$~

    Take~$Q$~satisfying~$1\ll Q\ll M.$~By Cauchy's inequality and Lemma~\ref{Fouvry-Iwaniec-lemma}, we have
    \begin{eqnarray*}
       |S_{II}|^2 & \ll & \bigg(\sum_{n\sim N}|b_n|^2\bigg)
        \bigg(\sum_{n\sim N}\bigg|\sum_{m\sim M}a_m e(xm^cn^c)\bigg|^2\bigg)   \\
             & \ll & N\sum_{n\sim N}\frac{M}{Q}\sum_{|q|<Q}\left(1-\frac{|q|}{Q}\right)
                        \sum_{M<m+q,m-q\leqslant2M} a_{m+q}\overline{a_{m-q}}e\big(xn^c\Delta_c(m,q)\big)  \\
             & \ll & \frac{M^2N^2}{Q} +\frac{MN}{Q} \sum_{1\leqslant q<Q}
                        \sum_{m\sim M}\bigg|\sum_{n\sim N} e\big(xn^c\Delta_c(m,q)\big) \bigg| ,
    \end{eqnarray*}
    where~$\Delta_c(m,q)=(m+q)^c-(m-q)^c.$~Thus, it is sufficient to estimate the following sum
    \begin{equation*}
         S_0:=\sum_{n\sim N}e\big(xn^c\Delta_c(m,q)\big).
    \end{equation*}
    By the method of exponent pairs, we get
    \begin{equation*}
         S_0 \ll \frac{MN}{Fq}+\left(\frac{Fq}{MN}\right)^\kappa N^\lambda,
    \end{equation*}
    where~$(\kappa,\lambda)$~is an arbitrary exponent pair. Therefore, we have
    \begin{eqnarray*}
        |S_{II}|^2 & \ll & \frac{M^2N^2}{Q}+\frac{MN}{Q}\sum_{1\leqslant q<Q}\sum_{m\sim M}
                             \left(\frac{MN}{Fq}+\left(\frac{Fq}{MN}\right)^\kappa N^\lambda  \right)    \\
                   & \ll & \frac{M^2N^2}{Q}+\frac{M^3N^2}{QF}\log Q+Q^\kappa F^\kappa M^{2-\kappa}N^{1+\lambda-\kappa}     \\
                   & \ll & \frac{M^2N^2}{Q}+Q^\kappa F^\kappa M^{2-\kappa}N^{1+\lambda-\kappa}     \\
                   & \ll & \frac{M^2N^2}{Q}+Q^\kappa F_0^\kappa M^{2-\kappa}N^{1+\lambda-\kappa}.
    \end{eqnarray*}
    Set
    \begin{equation*}
          Q_0=F_0^{-\kappa/(1+\kappa)}M^{\kappa/(1+\kappa)}N^{(1+\kappa-\lambda)/(1+\kappa)}.
    \end{equation*}
    Next, we will discuss three cases of the selection of~$Q.$~

    \textbf{Case 1} If~$Q_0<5,$~then we take~$Q=5$~and obtain
       \begin{equation*}
           |S_{II}|^2\ll F_0^\kappa M^{2-\kappa}N^{1+\lambda-\kappa}.
       \end{equation*}

    \textbf{Case 2} If~$5\leqslant Q_0\leqslant M/2,$~then we take~$Q=Q_0,$~and obtain
      \begin{equation*}
         |S_{II}|^2\ll F_0^{\kappa/(1+\kappa)} M^{2-\kappa/(1+\kappa)}N^{2-(1+\kappa-\lambda)/(1+\kappa)}.
      \end{equation*}

    \textbf{Case 3} If~$ Q_0>M/2,$~then we take~$Q=M/2,$~and obtain
      \begin{equation*}
           |S_{II}|^2\ll MN^2.
      \end{equation*}
Based on the above three cases, we have
    \begin{eqnarray}\label{Type-bound-2}
           S_{II} & \ll &  M^{1/2}N+F_0^{\kappa/2}M^{1-\kappa/2}N^{(1+\lambda-\kappa)/2}  \nonumber \\
                  &     &  +F_0^{\kappa/(2+2\kappa)}M^{1-\kappa/(2+2\kappa)}N^{1-(1+\kappa-\lambda)/(2+2\kappa)}.
    \end{eqnarray}
    According to~(\ref{Type-bound-1})~and~(\ref{Type-bound-2})~and noting
    that~$M^{1/2}N\asymp XM^{-1/2}\ll X^{1-36\delta/7}\ll X^{1-\delta},$~we get
    \begin{eqnarray*}
       S_{II} & \ll &  X^{1-\delta}+
                       F_0^{\kappa/(2+2\kappa)}M^{1-\kappa/(2+2\kappa)}N^{1-(1+\kappa-\lambda)/(2+2\kappa)}     \\
              &     &  +F_0^{\kappa/2}M^{1-\kappa/2}N^{(1+\lambda-\kappa)/2}+ M^{467/512}N^{65/64}F_0^{-1/128}.
    \end{eqnarray*}
    According to Lemma~\ref{Graham-Kolesnik}, there exists an~$F_0$~satisfying~$MN^2\ll F_0\ll KX^c$~such that
    \begin{eqnarray*}
        S_{II} & \ll & X^{1-\delta} +M^{467/512}N^{65/64}X^{-c/128}
                          +(MN^2)^{\kappa/2}M^{1-\kappa/2}N^{(1+\lambda-\kappa)/2}   \\
               &     &    +(MN^2)^{\kappa/(2+2\kappa)} M^{1-\kappa/(2+2\kappa)} N^{1-(1+\kappa-\lambda)/(2+2\kappa)}  \\
               &     &    +\big( (M^{1-\kappa/2}N^{(1+\lambda-\kappa)/2})^{1/128}(M^{467/512}N^{65/64})^{\kappa/2}
                          \big)^{1/(\kappa/2+1/128)}    \\
    \end{eqnarray*}
    \begin{eqnarray*}
               &     &    +\big((M^{1-\kappa/(2+2\kappa)} N^{1-(1+\kappa-\lambda)/(2+2\kappa)})^{1/128}  \\
               &     &    \quad\times(M^{467/512} N^{65/64})^{\kappa/(2+2\kappa)}\big)^{1/(1/128+\kappa/(2+2\kappa))}  \\
               &  =: &     X^{1-\delta} +\mathscr{J}_1+\mathscr{J}_2+\mathscr{J}_3+\mathscr{J}_4+\mathscr{J}_5.
    \end{eqnarray*}
    Taking~$(\kappa,\lambda)=ABABA^2B(0,1)=(1/11,3/4),$~then under the condition $X^{72\delta/7}\ll M \ll X^{1/2}$, we obtain
    \begin{equation*}
        \mathscr{J}_i\ll X^{1-\delta},\qquad i=1,2,3,4,5.
    \end{equation*}
    Therefore, we have
    \begin{equation*}
         S_{II}\ll X^{1-\delta}.
    \end{equation*}
 This completes the proof of Lemma~\ref{Type-II}.
\end{proof}

\begin{lemma}\label{S(x)estimate-2<c<2.025}
   Suppose~$2<c<37/18,$~then for~$\tau\leqslant|x|\leqslant K$~we have
    \begin{equation*}
          S(x) \ll X^{1-\delta}.
    \end{equation*}
\end{lemma}
\begin{proof}
   First, we have
   \begin{equation*}
       S(x)=U(x)+O(x^{1/2}),
   \end{equation*}
   where
   \begin{equation*}
       U(x)=\sum_{X/2<n\leqslant X}\Lambda(n)e(xn^c).
   \end{equation*}
   By Heath-Brown identity~\cite{Heath-Brown}~with~$k=3,$~it is easy to see that~$U(x)$~can be written as~$O(\log^6X)$~sums of the form
   \begin{equation*}
       U^*(x)=\sum_{n_1\sim N_1}\cdots\sum_{n_6\sim N_6}\log n_1\cdot\mu(n_4)\mu(n_5)\mu(n_6)e\big(x(n_1\cdots n_6)^c\big),
   \end{equation*}
   where~$N_1,\cdots,N_6\geqslant1,\,N_1\cdots N_6\asymp X,\,n_4,n_5,n_6\leqslant(2X)^{1/3}$~and some~$n_i$~may only take value~$1.$~

    Let~$F=|x|X^c.$~For~$2<c<37/18$, we shall prove that for each~$U^*(x)$~one has
    \begin{equation*}
         U^*(x)\ll X^{1-\delta}.
    \end{equation*}

    \textbf{Case 1} If there exists an~$N_j$~such that~$N_j\geqslant X^{1-72\delta/7}>X^{1/2},$~then we must have~$j\leqslant3.$~Take~$m=n_j,\,n=\prod\limits_{i\neq j}n_i,\, M=N_j,\,N=\prod\limits_{i\neq j}N_i.$~In this case, we can see that~$U^*(x)$~can be written as
    \begin{equation*}
         U^*(x)=\sum_{m\sim M}\sum_{n\sim N} a_mb_n e(xm^cn^c),
    \end{equation*}
    where~$|a_m|\leqslant\log m,\,|b_n|\leqslant d_5(n).$~Then~$U^*(x)$~is a sum of Type I. By Lemma~\ref{Type-I},~the
    result follows.

    \textbf{Case 2}  If there exists an~$N_j$~such that~$X^{72\delta/7}\leqslant N_j\leqslant X^{1-72\delta/7},$~then we
    take~$m=n_j,\, n=\prod\limits_{i\neq j}n_i,\, M^*=N_j,\, N^*=\prod\limits_{i\neq j}N_i.$~In this case, we can see that~$U^*(x)$~can be written as
   \begin{equation*}
       U^*(x)=\sum_{m\sim M^*}\sum_{n\sim N^*} a_mb_n e(xm^cn^c),
   \end{equation*}
   where~$|a_m|\leqslant\log m,\,|b_n|\leqslant d_5(n)\log n.$~
   If~$X^{72\delta/7}\leqslant M^* \leqslant X^{1/2},$~then~$N^* \gg X^{1/2}$~and we take~$(M,N)=(M^*,N^*).$
   If~$X^{72\delta/7}\leqslant N^* \leqslant X^{1/2},$~then~$M^* \gg X^{1/2}$~and we take~$(M,N)=(N^*,M^*).$
   Then~$U^*(x)$~is a sum
   of Type II. By Lemma~\ref{Type-II},~the result follows.

   \textbf{Case 3} If~$N_j<X^{72\delta/7} \,(j=1,2,3,4,5,6),$~without loss of generality, we assume
    that~$N_1\geqslant N_2\geqslant\cdots\geqslant N_6.$~Let~$\ell$~denote the smallest natural number~$j$~such that
    \begin{equation*}
       N_1N_2\cdots N_{j-1}<X^{72\delta/7},\quad\, N_1\cdots N_j\geqslant X^{72\delta/7},
    \end{equation*}
   then~$2\leqslant \ell\leqslant5.$~Noting that~$\delta<1/36<7/216,$~we obtain
   \begin{equation*}
        X^{72\delta/7}\leqslant N_1\cdots N_{\ell-1}\cdot N_\ell<X^{72\delta/7}\cdot X^{72\delta/7}<
        X^{1-72\delta/7}.
   \end{equation*}
   Let~$m=\prod\limits_{i=1}^\ell n_i,\,n=\prod\limits_{i=\ell+1}^6 n_i,\,M^*=\prod\limits_{i=1}^\ell N_i,\,N^*=\prod\limits_{i=\ell+1}^6 N_i.$~At this time, we can follow the discussion of Case 2 exactly and get the result by Lemma~\ref{Type-II}.  This completes the proof of Lemma~\ref{S(x)estimate-2<c<2.025}.
\end{proof}

\section{Proof of Theorem~\ref{Theorem-exceptional}}
Let us denote
\begin{equation*}
   \begin{array}{ll}
      H(R)=\displaystyle\int_{-\infty}^{+\infty}I^{3}(x)e(-Rx)\Phi(x)\mathrm{d}x,
      &  H_1(R)=\displaystyle\int_{-\tau}^{+\tau}I^{3}(x)e(-Rx)\Phi(x)\mathrm{d}x ,\\
   B_1(R)=\displaystyle\int_{-\infty}^{+\infty}S^{3}(x)e(-Rx)\Phi(x)\mathrm{d}x,
       &    D_1(R)= \displaystyle\int_{-\tau}^{+\tau}S^{3}(x)e(-Rx)\Phi(x)\mathrm{d}x, \\
   D_2(R)=\displaystyle\int_{\tau<|x|< K}S^{3}(x)e(-Rx)\Phi(x)\mathrm{d}x,
       &    D_3(R)=\displaystyle\int_{|x|\geqslant K}S^{3}(x)e(-Rx)\Phi(x)\mathrm{d}x.
   \end{array}
\end{equation*}
  In order to prove Theorem \ref{Theorem-exceptional}, it is sufficient to prove the following proposition.
\begin{proposition}\label{proposition-exceptional}
   Let $1<c<37/18,\,c\neq2.$ Then for any sufficiently large real number $N,$ we have
  \begin{equation}
     \int_{N}^{2N}|B_1(R)-H(R)|^2\mathrm{d}R\ll\varepsilon^2N^{6/c-1}
     \exp\left(-\frac{1}{3}\left(\frac{1}{c}\log\frac{2N}{3}\right)^{1/5}\right).
  \end{equation}
\end{proposition}
\subsection{Proof of Proposition~\ref{proposition-exceptional}}
    Throughout the proof of Proposition~\ref{proposition-exceptional}, we always set~$X=(2N/3)^{1/c}$~and denote the
    function~$\Phi(x)$~which is from Lemma~\ref{xiaobei} with parameters
\begin{equation*}
   a=\frac{9\varepsilon}{10},\qquad b=\frac{\varepsilon}{10},\qquad k=[\log X].
\end{equation*}

\noindent
We have
\begin{eqnarray}\label{Except-total}
  &     &   \!\!\!\!   \int_N^{2N}|B_1(R)-H(R)|^2\mathrm{d}R  \nonumber \\
  &   =  & \!\!\!\!  \int_N^{2N}|(D_1-H_1)+D_2+D_3-(H-H_1)|^2\mathrm{d}R   \nonumber \\
  & \ll & \!\!\!\!   \int_N^{2N}  \!\!\!|D_1-H_1|^2\mathrm{d}R+\int_N^{2N} \!\!\! |D_2|^2\mathrm{d}R
                              +\int_N^{2N} \!\!\! |D_3|^2\mathrm{d}R+\int_N^{2N}\!\!\! |H-H_1|^2\mathrm{d}R.
\end{eqnarray}
By Lemma \ref{Van der Corput}, we get $I(x)\ll X^{1-c}|x|^{-1}.$ By Lemma \ref{xiaobei}, we have
\begin{eqnarray} \label{Except-total-4}
   &     &  \int_{N}^{2N}|H-H_1|^2\mathrm{d}R
                 \nonumber \\
   & \ll &  \int_{N}^{2N}\bigg(\int_{|x|>\tau}|I(x)|^3 |\Phi(x)|\mathrm{d}x\bigg)^2\mathrm{d}R
                 \nonumber \\
   & \ll &  \varepsilon^2  N\bigg(\int_{|x|>\tau}|I(x)|^3\mathrm{d}x\bigg)^2
                 \nonumber \\
   & \ll &  \varepsilon^2 NX^{6-6c}\bigg(\int_{\tau}^{\infty}\frac{\mathrm{d}x}{x^3}\bigg)^2
            \ll \varepsilon^2 N \frac{X^{6-6c}}{\tau^4}.
\end{eqnarray}

  For the third term on the right hand in (\ref{Except-total}), we have
\begin{eqnarray}\label{Except-total-3}
  \int_N^{2N}|D_3|^2\mathrm{d}R
      & \ll & \int_{N}^{2N}\bigg(\int_K^{+\infty}|S(x)|^3|\Phi(x)|\mathrm{d}x\bigg)^2 \mathrm{d}R
                       \nonumber \\
      & \ll & N\Bigg(\int_K^{+\infty}|S(x)|^3|\Phi(x)|\mathrm{d}x\Bigg)^2
                       \nonumber \\
      & \ll & NX^6\bigg(\int_K^{+\infty}\bigg(\frac{5k}{\pi x\varepsilon}\bigg)^k\frac{\mathrm{d}x}{x}\bigg)^2
                       \nonumber \\
      & \ll & NX^6\bigg(\frac{5k}{\pi K\varepsilon}\bigg)^{2k}
               \ll N\frac{X^{6+2\log(5/\pi)}}{X^{4\log X}} \ll N.
\end{eqnarray}

  Take $\Omega_1=\{R:N<R\leqslant 2N\},\,\Omega_2=\{x:\tau<|x|<K\},\, \xi=S^3(x)\Phi(x),\, \omega(x,R)=e(Rx),\,c(R)=\overline{D_2(R)}.$ Then from Lemma \ref{Max-value-lemma} and Lemma \ref{Functional-lemma}, we obtain
\begin{eqnarray}
   \int_{N}^{2N}|D_2(R)|^2\mathrm{d}R & \leqslant & 2A\int_{\tau}^{K}|S(x)|^6|\Phi(x)|^2\mathrm{d}x
                 \nonumber \\
   & \ll & \mathscr{L}\cdot\max_{\tau\leqslant x\leqslant K}|S(x)|^2
           \times\int_{\tau}^K|S(x)|^4|\Phi(x)|^2\mathrm{d}x.
\end{eqnarray}
By the first derivative test, we have
\begin{eqnarray}
    &     &   \int_{\tau}^K |S(x)|^4|\Phi(x)|^2\mathrm{d}x
               \ll   \varepsilon^2\int_{\tau}^K |S(x)|^4\mathrm{d}x
                    \nonumber \\
    &  =  &   \varepsilon^2\sum_{\frac{X}{2}<p_1,p_2,p_3,p_4\leqslant X}(\log p_1)\cdots(\log p_4)
              \int_\tau^Ke\left((p_1^c+p_2^c-p_3^c-p_4^c)x\right)\mathrm{d}x
                      \nonumber \\
    & \ll &   \varepsilon^2 \log^4X \sum_{\frac{X}{2}<p_1,p_2,p_3,p_4\leqslant X}
              \min\bigg(K,\frac{1}{|p_1^c+p_2^c-p_3^c-p_4^c|}\bigg)
                      \nonumber \\
    & \ll &   \sum_{\frac{X}{2}<n_1,n_2,n_3,n_4\leqslant X}\min\bigg(K,\frac{1}{|n_1^c+n_2^c-n_3^c-n_4^c|}\bigg).
\end{eqnarray}
Let $u=n_1^c+n_2^c-n_3^c-n_4^c.$ By Lemma \ref{Robert-lemma}, the contribution of $K$ is
(notice $|u|\leqslant K^{-1}$)
\begin{equation}\label{3-theorem-K-contribution}
    \ll K\cdot\mathscr{A}(X/2;c,K^{-1})\ll (X^{4-c}+X^2)X^{\eta}.
\end{equation}
By a dyadic argument, the contribution from $n_1,\,n_2,\,n_3,\,n_4$ with $|u|>K^{-1}$ is bounded by
\begin{eqnarray}\label{3-theorem-K-except-contribution}
    &  \ll  &   \log X\times \max_{K^{-1}\leqslant U\ll X^c}
                \sum_{\substack{\frac{X}{2}<n_1,n_2,n_3,n_4\leqslant X\\ U<|u|\leqslant 2U }}\frac{1}{|u|}
                       \nonumber \\
    & \ll &   \log X\times \max_{K^{-1}\leqslant U\ll X^c}
              U^{-1} \cdot\mathscr{A}(X/2;c,2U)
                        \nonumber \\
    & \ll &   \log X\times \max_{K^{-1}\leqslant U\ll X^c} (X^{4-c}+X^2U^{-1})X^\eta
                         \nonumber \\
    & \ll &   (X^{4-c}+X^2)X^\eta.
\end{eqnarray}
Combining (\ref{3-theorem-K-contribution}) and (\ref{3-theorem-K-except-contribution}), we have
\begin{equation}\label{3.9}
  \int_{\tau}^{K}|S(x)|^4|\Phi(x)|^2\mathrm{d}x\ll (X^{4-c}+X^2)X^\eta.
\end{equation}
If $1<c<2$,  then from Lemma \ref{S(x)estimate-1<c<2} we get
\begin{eqnarray}\label{D_2-case-1}
   \int_{N}^{2N}|D_2(R)|^2\mathrm{d}R
   & \ll & \mathscr{L}\cdot\big(X^{(6+c)/4+\eta}+X^{28/15+\eta}\big)(X^{4-c}+X^2)X^\eta
                  \nonumber \\
   & \ll & \big(X^{(6+c)/4+\eta}+X^{28/15+\eta}\big)X^{4-c+\eta}
                  \nonumber \\
   & \ll & X^{11/2-3c/4+\eta}+X^{88/15-c+\eta}
                    \nonumber \\
   & \ll & N^{11/(2c)-3/4+\eta}+N^{88/(15c)-1+\eta}
                    \nonumber \\
   & \ll & \varepsilon^2N^{6/c-1}\mathscr{L}^{6}E^{2/3}  .
\end{eqnarray}
If~$2<c<37/18,$~then then from Lemma \ref{S(x)estimate-2<c<2.025} we get
\begin{eqnarray}\label{D_2-case-2}
                 \int_{N}^{2N}|D_2(R)|^2\mathrm{d}R
       & \ll &   \mathscr{L}\cdot X^{4-c-2\eta}(X^{4-c}+X^2)X^\eta
                     \nonumber \\
       & \ll &   \mathscr{L}\cdot X^{4-c-2\eta}\cdot X^{2+\eta} \ll X^{6-c-\eta}
                     \nonumber \\
       & \ll &    N^{6/c-1-\eta} \ll \varepsilon^2N^{6/c-1}\mathscr{L}^{6}E^{2/3}.
\end{eqnarray}
Combining (\ref{D_2-case-1}) and (\ref{D_2-case-2}), for~$1<c<37/18,\,c\neq2,$~we obtain that
\begin{equation}\label{Except-total-2}
    \int_{N}^{2N}|D_2(R)|^2\mathrm{d}R \ll \varepsilon^2N^{6/c-1}\mathscr{L}^{6}E^{2/3}.
\end{equation}
Next, we consider the first term on the right hand in (\ref{Except-total}). First of all, one has
\begin{eqnarray*}
   &    &   |D_1(R)-H_1(R)|^2   \nonumber \\
   &  = &   \int_{-\tau}^\tau\left(\overline{S(x)}^3-\overline{I(x)}^3\right)e(Rx)\overline{\Phi(x)}\mathrm{d}x
          \times \int_{-\tau}^\tau \left(S^3(y)-I^3(y)\right)e(-R y)\Phi(y)\mathrm{d}y
                          \nonumber \\
   & = &  \int_{-\tau}^\tau\left(\overline{S(x)}^3-\overline{I(x)}^3\right)\overline{\Phi(x)}
         \bigg(\int_{-\tau}^\tau\left(S^3(y)-I^3(y)\right)e\big(R(x- y)\big)\Phi(y)\mathrm{d}y\bigg)\mathrm{d}x.
\end{eqnarray*}
Therefore, we have
\begin{eqnarray} \label{first-term-1}
 &     &     \int_{N}^{2N}|D_1(R)-H_1(R)|^2\mathrm{d}R   \nonumber \\
 & = &     \int_{N}^{2N}\bigg[\int_{-\tau}^\tau\left(\overline{S(x)}^3-\overline{I(x)}^3\right)\overline{\Phi(x)}
                     \nonumber \\
   &  & \qquad\times \bigg(\int_{-\tau}^\tau\left(S^3(y)-I^3(y)\right)e\big(R(x- y)\big)\Phi(y)\mathrm{d}y\bigg)\mathrm{d}x
               \bigg]\mathrm{d}R
                       \nonumber \\
 & = & \int_{-\tau}^\tau\left(\overline{S(x)}^3-\overline{I(x)}^3\right)\overline{\Phi(x)}
                      \nonumber \\
 &   & \qquad\times\bigg[\int_{-\tau}^\tau\left(S^3(y)-I^3(y)\right)\bigg(\int_{N}^{2N}e(R(x- y))
           \mathrm{d}R\bigg)\Phi(y)\mathrm{d}y\bigg]\mathrm{d}x
                       \nonumber \\
 & \ll &  \int_{-\tau}^\tau|S^3(x)-I^3(x)||\Phi(x)|
                       \nonumber \\
 &  & \qquad\times \bigg(\int_{-\tau}^\tau|S^3(y)-I^3(y)||\Phi(y)|\min\bigg(N,\frac{1}{|x-y|}\bigg)\mathrm{d}y
         \bigg)\mathrm{d}x.
\end{eqnarray}

   Applying Cauchy's inequality to the inner integral and combining Lemma \ref{S(x),I(x)-4-mean-square}, one has
\begin{eqnarray}\label{first-term-2}
   &   &    \int_{-\tau}^\tau\left|S^3(y)-I^3(y)\right||\Phi(y)|\min\bigg(N,\frac{1}{|x-y|}\bigg)\mathrm{d}y
                   \nonumber \\
   & \ll &  \varepsilon \int_{-\tau}^\tau\left|S(y)-I(y)\right|\left|S^2(y)+S(y)I(y)+I^2(y)\right|
             \min\bigg(N,\frac{1}{|x-y|}\bigg)\mathrm{d}y
                     \nonumber \\
   & \ll &  \varepsilon\bigg(\int_{-\tau}^\tau \left|S^2(y)+S(y)I(y)+I^2(y)\right|^2\mathrm{d}y \bigg)^{1/2}
                     \nonumber \\
   &     &  \qquad \times \bigg( \int_{-\tau}^\tau\left|S(y)-I(y)\right|^2
              \min\bigg(N,\frac{1}{|x-y|}\bigg)^2\mathrm{d}y  \bigg)^{1/2}
                     \nonumber \\
   & \ll &  \varepsilon \bigg(\int_{-\tau}^\tau|S(y)|^4 \mathrm{d}y +\int_{-\tau}^\tau|I(y)|^4
                   \mathrm{d}y  \bigg)^{1/2}
                      \nonumber \\
   &     &  \qquad \times \bigg( \int_{-\tau}^\tau\left|S(y)-I(y)\right|^2
              \min\bigg(N,\frac{1}{|x-y|}\bigg)^2\mathrm{d}y  \bigg)^{1/2}
                       \nonumber \\
   &  \ll & \varepsilon X^{2-c/2}\mathscr{L}^{5/2}\bigg( \int_{-\tau}^\tau\left|S(y)-I(y)\right|^2
              \min\bigg(N,\frac{1}{|x-y|}\bigg)^2\mathrm{d}y  \bigg)^{1/2}.
\end{eqnarray}
Put (\ref{first-term-2}) into (\ref{first-term-1}) and we get
\begin{eqnarray}\label{first-term-3}
    &     &    \int_{N}^{2N}|D_1(R)-H_1(R)|^2\mathrm{d}R
                        \nonumber \\
    & \ll &   \varepsilon^{3/2}X^{2-c/2}\mathscr{L}^{5/2}\int_{-\tau}^{\tau}|S^3(x)-I^3(x)||\Phi(x)|^{1/2}
                       \nonumber \\
    &     &   \qquad \times \bigg( \int_{-\tau}^\tau\left|S(y)-I(y)\right|^2
              \min\bigg(N,\frac{1}{|x-y|}\bigg)^2\mathrm{d}y  \bigg)^{1/2}\mathrm{d}x
                       \nonumber \\
    & \ll &   \varepsilon^{3/2}X^{2-c/2}\mathscr{L}^{5/2}
              \sup_{|x|\leqslant \tau}\bigg(\int_{-\tau}^\tau\left|S(y)-I(y)\right|^2
               \min\bigg(N,\frac{1}{|x-y|}\bigg)^2|\Phi(x)|\mathrm{d}y\bigg)^{1/2}
                       \nonumber \\
    &     &    \qquad \times \int_{-\tau}^\tau |S^3(x)-I^3(x)|\mathrm{d}x.
\end{eqnarray}
On one hand, by Lemma \ref{square-mean-value}, we have
\begin{eqnarray}\label{first-term-4}
     \int_{-\tau}^{\tau} |S^3(x)-I^3(x)|\mathrm{d}x
  &  \leqslant &    \int_{-\tau}^{\tau} |S(x)|^3 \mathrm{d}x+ \int_{-\tau}^{\tau} |I(x)|^3 \mathrm{d}x
                         \nonumber \\
  & \ll &    X\int_{-\tau}^{\tau} |S(x)|^2\mathrm{d}x+ X\int_{-\tau}^{\tau} |I(x)|^2 \mathrm{d}x
                         \nonumber \\
  & \ll &     X^{3-c}\mathscr{L}^3.
\end{eqnarray}
On the other hand, by Lemma \ref{S(x)=I(x)+error}, we have
\begin{eqnarray}\label{first-term-5}
   &     &   \int_{-\tau}^{\tau} \left|S(y)-I(y)\right|^2 \min\bigg(N,\frac{1}{|x-y|}\bigg)^2|\Phi(x)|\mathrm{d}y
                             \nonumber \\
   & \ll &   \varepsilon N^2\int_{y\in\left(x-\frac{P}{N},\,x+\frac{P}{N}\right)\cap[-\tau,\tau]}
             \left|S(y)-I(y)\right|^2\mathrm{d}y
     +\varepsilon\frac{N^2}{P^2}\int_{-\tau}^{\tau}\left|S(y)-I(y)\right|^2\mathrm{d}y
                             \nonumber \\
   & \ll &   \varepsilon N^2X^{2-c}PE^2+\varepsilon\frac{N^2}{P^2}X^{2-c}\mathscr{L}^3
                              \nonumber \\
   & \ll &   \varepsilon N^2X^{2-c}E^{4/3}\mathscr{L}^3.
\end{eqnarray}
Combining (\ref{first-term-3}), (\ref{first-term-4}) and (\ref{first-term-5}) we obtain
\begin{equation}\label{Except-total-1}
    \int_{N}^{2N}|D_1(R)-H_1(R)|^2\mathrm{d}R   \ll\varepsilon^2 N^{6/c-1}\mathscr{L}^7 E^{2/3}.
\end{equation}
 From (\ref{Except-total}), (\ref{Except-total-4}), (\ref{Except-total-3}), (\ref{Except-total-2}) and
  (\ref{Except-total-1}), we know that the conclusion of Proposition \ref{proposition-exceptional} follows.

   \subsection{Proof of Theorem~\ref{Theorem-exceptional}}
For~$R\in(N,2N],$~we set
\begin{equation*}
   B:=B(R)=\sum_{\substack{\frac{X}{2}<p_1,p_2,p_3\leqslant X\\|p_1^c+p_2^c+p_3^c-R|<\varepsilon}}
   (\log p_1)(\log p_2)(\log p_3).
\end{equation*}
From Proposition~\ref{proposition-exceptional}, we can claim that if~$1<c<37/18,\,c\neq2,$~there exists a set~$\mathfrak{P}\subset(N,2N]$~satisfying
 \begin{equation}\label{ex-p-bound}
      |\mathfrak{P}|\ll N\exp\bigg(-\frac{2}{15}\bigg(\frac{1}{c}\log\frac{2N}{3}\bigg)^{1/5}\bigg),
  \end{equation}
such that
\begin{equation*}
   B_1(R)=H(R)+O\Bigg(\varepsilon N^{\frac{3}{c}-1}\exp\bigg(-\frac{1}{10}\bigg(\frac{1}{c}\log\frac{2N}{3}\bigg)^{1/5}\bigg)\Bigg)
\end{equation*}
for all~$R\in (N,2N]\setminus\mathfrak{P}.$~

  Actually, from Proposition~\ref{proposition-exceptional}, for~$R\in\mathfrak{P},$~we have
\begin{equation}  \label{B_1(R)-H(R)}
   B_1(R)-H(R)\gg\varepsilon
          N^{\frac{3}{c}-1}\exp\bigg(-\frac{1}{10}\bigg(\frac{1}{c}\log\frac{2N}{3}\bigg)^{1/5}\bigg).
\end{equation}
 Therefore, we get
 \begin{eqnarray*}
     &  &  \varepsilon^2N^{6/c-1}
           \exp\left(-\frac{1}{3}\left(\frac{1}{c}\log\frac{2N}{3}\right)^{1/5}\right)   \nonumber \\
       & \gg & \int_{N}^{2N} |B_1(R)-H(R)|^2\mathrm{d}R   \nonumber \\
       & \gg & \int_{\mathfrak{P}} |B_1(R)-H(R)|^2\mathrm{d}R   \nonumber \\
       & \gg & |\mathfrak{P}|\cdot
                \varepsilon^2N^{6/c-2}
                \exp\left(-\frac{1}{5}\left(\frac{1}{c}\log\frac{2N}{3}\right)^{1/5}\right),
\end{eqnarray*}
 and~(\ref{ex-p-bound})~follows.

     As in \cite{D. I. Tolev}, by the Fourier transformation formula, we have
  \begin{eqnarray*}
     B_1(R) & = & \sum_{\frac{X}{2}<p_1,\,p_2,\,p_3\leqslant X}\log p_1\cdot\log p_2\cdot\log p_3
                  \cdot\int_{-\infty}^{\infty}e\big((p_1^c+p_2^c+p_3^c-R)x\big)\Phi(x)\mathrm{d}x      \\
     & = & \sum_{\frac{X}{2}<p_1,\,p_2,\,p_3\leqslant X}\log p_1\cdot\log p_2\cdot\log p_3
           \cdot \varphi\left(p_1^c+p_2^c+p_3^c-R\right)  \leqslant B(R).
  \end{eqnarray*}

 Hence Theorem \ref{Theorem-exceptional} follows from the inequality
 \begin{equation*}
    H(R)\gg\varepsilon R^{3/c-1},
 \end{equation*}
which can be proved proceeding as in \cite{D. I. Tolev}, Lemma 6. This completes the proof of Theorem~\ref{Theorem-exceptional}.

\section{Proof of Theorem~\ref{Theorem-1} }

  Throughout the proof of Theorem~\ref{Theorem-1}, we  we always set~$X=(N/5)^{1/c}$~and denote the
  function~$\varphi(y)$~which is from Lemma \ref{xiaobei} with
parameters
\begin{equation*}
   a=\frac{9\varepsilon}{10},\qquad b=\frac{\varepsilon}{10},\qquad k=[\log X].
\end{equation*}
Let
\begin{equation*}
   \mathscr{B}=\sum_{\substack{\frac{X}{2}<p_1,\,\cdots\,,p_6\leqslant X \\ |p_1^c+\cdots+p_6^c-N|<\varepsilon}}
   (\log p_1)(\log p_2)\cdots(\log p_6).
\end{equation*}
Set
\begin{equation*}
   \mathscr{B}_1=\sum_{\frac{X}{2}<p_1,p_2,\cdots,p_6\leqslant X}(\log p_1)(\log p_2)\cdots(\log p_6)
   \cdot\varphi(p_1^c+\cdots+p_6^c-N).
\end{equation*}
By the definition of~$\varphi,$~we have
\begin{equation}\label{B>B_1}
       \mathscr{B}\geqslant  \mathscr{B}_1.
\end{equation}
The Fourier transformation formula gives
\begin{eqnarray}\label{B_1-fenjie}
    \mathscr{B}_1 & = & \int_{-\infty}^{+\infty}S^6(x)e(-Nx)\Phi(x)\mathrm{d}x
                            \nonumber  \\
                  & =: & \mathscr{D}_1+\mathscr{D}_2+\mathscr{D}_3,
\end{eqnarray}
where
\begin{eqnarray*}
   &  &  \mathscr{D}_1=\int_{-\tau}^{\tau}    S^6(x)e(-Nx)\Phi(x)\mathrm{d}x, \\
   &  &  \mathscr{D}_2=\int_{\tau<|x|<K}      S^6(x)e(-Nx)\Phi(x)\mathrm{d}x, \\
   &  &  \mathscr{D}_3=\int_{|x|\geqslant K}  S^6(x)e(-Nx)\Phi(x)\mathrm{d}x.
\end{eqnarray*}
By Lemma \ref{xiaobei}, we have
\begin{equation}\label{D_3-upper-bound}
   \mathscr{D}_3\ll\int_{K}^{+\infty}|S(x)|^6|\Phi(x)|\mathrm{d}x\ll X^6\int_{K}^{+\infty}\frac{1}{x}
   \left(\frac{5k}{\pi x\varepsilon}\right)^k\mathrm{d}x\ll1.
\end{equation}
Let
 $$\mathscr{H}_1=\int_{-\tau}^{\tau}I^{6}(x)e(-Nx)\Phi(x)\mathrm{d}x$$
  and
  $$\mathscr{H}=\int_{-\infty}^{+\infty}I^{6}(x)e(-Nx)\Phi(x)\mathrm{d}x,$$
  then
\begin{equation} \label{D_1-fenjie}
  \mathscr{D}_1=\mathscr{H}+(\mathscr{H}_1-\mathscr{H})+(\mathscr{D}_1-\mathscr{H}_1).
\end{equation}
By Lemma \ref{Van der Corput}, we get $I(x)\ll X^{1-c}|x|^{-1}.$ By Lemma~\ref{xiaobei}, we have
\begin{eqnarray}\label{H_1-H}
   \mathscr{H}_1-\mathscr{H} & \ll & \int_{\tau}^{+\infty}|I(x)|^6|\Phi(x)|\mathrm{d}x
                    \nonumber \\
         & \ll & X^{6-6c}\int_{\tau}^{+\infty}\frac{\mathrm{d}x}{x^7}
                     \nonumber \\
         & \ll & X^{6-6c}\tau^{-6}\ll X^{6\eta}.
\end{eqnarray}
According to Lemma~\ref{xiaobei},~Lemma~\ref{square-mean-value}~and Lemma~\ref{S(x)ti-huan-I(x)}, we have
\begin{eqnarray}\label{D_1-H_1}
    \mathscr{D}_1-\mathscr{H}_1 & \ll &   \int_{-\tau}^{\tau}\left|S^6(x)-I^6(x)\right|\left|\Phi(x)\right|\mathrm{d}x
                           \nonumber \\
            & \ll &   \varepsilon \cdot\max_{|x|\leqslant\tau}\left|S(x)-I(x)\right|\times\int_{-\tau}^{\tau}
                      \left(\left|S(x)\right|^5+\left|I(x)\right|^5\right)\mathrm{d}x
                             \nonumber \\
            & \ll &   \varepsilon X^3\cdot\max_{|x|\leqslant\tau}\left|S(x)-I(x)\right|\times\int_{-\tau}^{\tau}
                      \left(\left|S(x)\right|^2+\left|I(x)\right|^2\right)\mathrm{d}x
                           \nonumber \\
            & \ll &   \varepsilon X^{5-c}\log^3X\times\max_{|x|\leqslant\tau}\left|S(x)-I(x)\right|
                            \nonumber \\
            & \ll &    \varepsilon X^{6-c}e^{-\frac{1}{2}\left(\log X\right)^{1/5}}.
\end{eqnarray}
So Lemma \ref{6-power-zhuxiang-low-bound} combining (\ref{D_1-fenjie}), (\ref{H_1-H}) and (\ref{D_1-H_1}) yields
\begin{equation}\label{D_1-low-bound}
   \mathscr{D}_1\gg \varepsilon X^{6-c}.
\end{equation}
 For $\mathscr{D}_2$, we have
\begin{equation}\label{D_2-estimate}
   \mathscr{D}_2\ll \max_{\tau<|x|<K}\left|S(x)\right|^2\times
                \int_{\tau<|x|<K}\left|S(x)\right|^4\left|\Phi(x)\right|\mathrm{d}x.
\end{equation}
For the integral on the right hand in (\ref{D_2-estimate}), we can exactly follow the process of~(\ref{3.9}) and obtain
\begin{equation}\label{D-2-S^4-Phi}
   \int_{\tau<|x|<K}\left|S(x)\right|^4\left|\Phi(x)\right|\mathrm{d}x
    \ll(X^{4-c}+X^2)X^{3\eta}.
\end{equation}
If $1<c<2,$ then from (\ref{D_2-estimate}), Lemma \ref{S(x)estimate-1<c<2} and (\ref{D-2-S^4-Phi}) we get
\begin{equation}\label{D_2-upper-1}
   \mathscr{D}_2\ll (X^{(6+c)/4+\eta}+X^{28/15+\eta})(X^{4-c}+X^{2})X^\eta\ll X^{6-c-\eta}.
\end{equation}
If $2<c<37/18,$ then from (\ref{D_2-estimate}), Lemma \ref{S(x)estimate-2<c<2.025} and (\ref{D-2-S^4-Phi}) we get
\begin{equation}\label{D_2-upper-2}
    \mathscr{D}_2\ll X^{2-2\delta}(X^{4-c}+X^{2})X^\eta\ll X^{2-2\delta}\cdot X^{2+\eta}\ll  X^{6-c-\eta} .
\end{equation}
From (\ref{B>B_1}), (\ref{B_1-fenjie}), (\ref{D_3-upper-bound}), (\ref{D_1-low-bound}), (\ref{D_2-upper-1}) and (\ref{D_2-upper-2}) we get
\begin{equation*}
    \mathscr{B}\geqslant \mathscr{B}_1=\mathscr{D}_1+\mathscr{D}_2+\mathscr{D}_3\gg\varepsilon X^{6-c},
\end{equation*}
which completes the proof of Theorem~\ref{Theorem-1} .

\bigskip
\bigskip

\textbf{Acknowledgement}

   The authors would like to express the most and the greatest sincere gratitude to Professor Wenguang Zhai for his valuable
advice and constant encouragement.

\end{document}